\documentclass[review]{elsarticle} 

\usepackage{float}
\usepackage{amsmath,amssymb,amsfonts,amstext}
\usepackage{epsfig}
\usepackage{setspace}
\usepackage{hhline}
\usepackage{color,soul}
\usepackage[dvipsnames]{xcolor}
\usepackage[justification=centering]{caption}
\usepackage{multirow}
\usepackage{booktabs}
\usepackage{bm}
\usepackage[english]{babel}
\usepackage{tikz}
\usepackage{tikzscale}
\usepackage{epstopdf}
\usepackage{adjustbox}
\usepackage{longtable}
\usepackage{accents}
\usepackage{lineno}
\usepackage{colortbl}
\usepackage[linesnumbered,ruled,vlined]{algorithm2e}
\usepackage{ragged2e}
\makeatletter
\usepackage{breqn}
\usepackage{amsthm}
\usepackage{pdflscape}
\usepackage{afterpage}
\usepackage{capt-of}
\usepackage{dsfont}
\usepackage{gensymb}
\usepackage[mathscr]{euscript} 
\usepackage{mathtools}

\newcolumntype{x}[1]{>{\raggedright\arraybackslash\hspace{0pt}}p{#1}}
\newtheorem{problem}{Problem}
\newtheorem{assumption}{Assumption}
\newtheorem{definition}{Definition}
\newtheorem{lemma}{Lemma}
\newtheorem{theorem}{Theorem}
\newtheorem{remark}{Remark}

\newcommand{\ceil}[1]{\left\lceil #1 \right\rceil}

\newcommand{\iter}[1]{\ensuremath{\langle #1 \rangle}}
\newcommand{\itrpt}[3]{\ensuremath{\bm{#1}^{\iter{#2}}_{#3}}}

\newcommand{\idxpt}[3]{\ensuremath{\bm{#1}^{(#2)}_{#3}}}
\newcommand{\idxptT}[3]{\ensuremath{\bm{#1}^{(#2)\top}_{#3}}}
\newcommand{\idxval}[3]{\ensuremath{#1^{(#2)}_{#3}}}

\newcommand{\titrpt}[3]{\ensuremath{\tilde{\bm{#1}}^{\iter{#2}}_{#3}}}
\newcommand{\triple}[3]{\ensuremath{\mathrm{col}\!\,(#1,\, #2,\, #3)}}

\newcommand{\admtriplesset}{\mathcal{T}'{}^{\iter{n}}}
\SetCommentSty{mycommfont}
\SetKwInput{KwInput}{Input}        
\SetKwInput{KwOutput}{Output}       

\DeclareMathOperator*{\argmin}{arg\,min}

\begin{document}

\begin{frontmatter}
\title{\underline{D}irect \underline{I}nference \underline{on} \underline{A}dmissible Sets (\textsc{Diona})}
\author{Lorenzo Sabug Jr.\corref{corr}}\ead{l.sabug21@imperial.ac.uk}
\author{Eric Kerrigan}\ead{e.kerrigan@imperial.ac.uk}

\cortext[corr]{Corresponding author. A preliminary version of this work was presented at the 23rd IFAC World Congress 2026, Busan, Republic of Korea. This work was supported by the UKRI EPSRC Grant EP/Z533816/1 ``Concurrent Learning and Control of Uncertain Large-Scale Phenomena''.}

\begin{abstract}
We propose a new regression method named Direct Inference on Admissible Sets (\textsc{Diona}) that directly uses samples to estimate the state and its derivatives at a query point, along with their admissible bounds, without long global training. By casting the values and derivatives at the query and neighbouring samples as latent variables in a constrained optimisation problem, the burden of long training loops becomes one of on-demand local solves. We establish \textsc{Diona}'s theoretical properties including finite-sample coverage and, in the noiseless case, convergence to the true state and derivatives under dense sampling. In a nonlinear regression benchmark on six functions of three to eight dimensions, \textsc{Diona} was more accurate than a Gaussian process (GP) on noiseless data especially with denser data sets. Furthermore, as with GP, it produced uncertainty bounds that contained the true value at more than 95\% of query points in four functions. \textsc{Diona}'s time to a first prediction after a data set update was three to four orders of magnitude shorter than that of the GP, and two to three orders shorter than that of a neural network, which was more accurate but offers no uncertainty quantification. This low latency is paid for by on-demand predictions, so this latency advantage applies when relatively few predictions are needed between data updates.
\end{abstract}
    
\end{frontmatter}

\section{Introduction}

The reconstruction of a hidden function from existing samples is a long-standing problem across science and engineering, especially when acquiring samples proves expensive and/or challenging, e.g., in computational fluid dynamics, wireless sensor networks-based data gathering, or robot-based exploration. Such a problem, referred to as ``data-driven regression'' or ``interpolation'', is a well-investigated topic \cite{theodoridis2020machine}, addressing the need for fast, accurate, and reliable estimates.

There have been many previous works on data-driven regression, with classical methods relying on basis function fitting \cite{martin2024polynomial,trefethen2013approximation}, kernel functions \cite{buhmann2003radial,rasmussen2006gaussian}, training parametric surrogate models, e.g., neural networks \cite{bishop1995neural}. However, there is need to choose/build a function library, kernel, and/or an architecture, which is consequential for the resulting accuracy and computational burden. Another class of methods rely on direct samples weighting \cite{roll2005direct,lu2008adaptive}, which however do not account for quadratic and/or curvature-related trends, affecting prediction accuracy for reconstructing nonlinear functions from sparse samples. Other approaches do not assume nor fit an \textit{a priori} known/designated function form, and choose their estimates from an admissible set. Examples include nonlinear set membership (NSM) methods \cite{milanese2004setmembership,milanese2011unified}, Lipschitz-based interpolation \cite{huang2023lipschitz}, and kinky inference \cite{calliess2020lacki}, whose heritage we take for our proposed method. Aside from point-value estimates, uncertainty quantification is investigated in different forms in \cite{milanese2004setmembership,calliess2020lacki,rasmussen2006gaussian,martin2024polynomial}, among others, a relevant quantity for exploration-based sampling. 

In many compelling applications like emergency response and robotic exploration, rapidly accumulated samples and evolving data sets must influence the next decisions, e.g., to take an action or select the next sampling points. For exploration in particular, the uncertainty measure is also a mission decision metric, rather than a \textit{post hoc} reported quantity. Therefore, there is demand for approaches that balance accurate predictions, informative and tight uncertainty bounds, and computational efficiency: in particular, a low latency between a data set update and a prediction.

This paper proposes a new formulation to use local information to estimate the state at a prediction point (also referred to in this paper as `query point'). Without assuming a functional form nor parametrised model, we use regularity assumptions on the underlying functions and draw the Taylor relationships between the query point and neighbouring samples, computing the admissible set of state values and derivatives. Our method builds on NSM, kinky inference, and Lipschitz interpolation, but they only exploit a first-order regularity assumption, so their uncertainty bounds grow linearly with distance from the closest sample. Our method, on the other hand, uses a Lipschitz-continuous Hessian with latent gradients and Hessians; these account for the local curvature leading to uncertainty that shrinks with the cube of the sample neighbourhood radius in the noiseless case. Our contributions are the following:

\begin{itemize}
    \item We cast the regression problem as a joint set estimation of admissible values and derivatives, using only (zeroth-order) function value samples, and regularity assumptions. As a result, we arrive at a point estimate, and the uncertainty bounds on the state and its derivatives.
    \item We derive the information radius---not just on the value estimate at the query point, but also its derivatives---for the function smoothness class we consider. Furthermore, in the noiseless case, we establish consistency on the value and derivatives as the samples set becomes dense.
    \item We use our results to build a practical regression algorithm which, given a samples set and a query point, solves small linear programs to simultaneously estimate the state and derivatives, as well as their corresponding uncertainty bounds. Our new Direct Inference on Admissible Sets (\textsc{Diona}) needs no offline training: every prediction is a per-query, on-demand solve, and therefore does not require any retraining in face of evolving (appended, changed) data sets.
\end{itemize}

We tested \textsc{Diona} in a nonlinear regression benchmark involving six functions of three to eight dimensions against a Gaussian process (GP), a neural network (NN), nonlinear set membership (NSM), radial basis functions and inverse distance weighting. In the noiseless cases, \textsc{Diona} had lower errors than the GP especially with the densest data sets. It also gave tighter uncertainty bounds than NSM while covering the true value within those bounds at more than 95\% of query points in four functions, similar to GP. In noisy cases, \textsc{Diona}'s accuracy advantages were swept away, although it remained competitive with the other methods. Furthermore, in the time that the GP needs to train, \textsc{Diona} produced around $10^4$ predictions ($10^2$--$10^3$ predictions within the training time of a GPU-trained NN). \textsc{Diona}'s timing advantage pays off in regimes where data sets change often and relatively few predictions are needed between updates.

A preliminary version of this work appeared as Direct Constraints Based Regression (DCBR) in \cite{sabug2026rethinking}. In relation to that work, we have significantly simplified the formulation, treated uncertainty quantification, investigated the theoretical properties, discussed implementation-related aspects, and performed a more exhaustive and detailed benchmark. The pertinent code base for \textsc{Diona}, along with the benchmark scripts, is available on GitHub under the following URL: \verb|https://github.com/lorenzosabugjr/Diona|.

\subsection*{Notation}

In this paper, we denote $\bm{x}' \in \mathbb{R}^D$ ($D$ is the dimensionality) as the query point, at which we estimate a state value $u'$ from existing samples and other \textit{a priori} information. We use ``sample'' to refer to the tuple $(\idxpt{x}{i}{}, \idxval{z}{i}{})$, where $\idxpt{x}{i}{} \in \mathbb{R}^D$ is the sampled location, and $\idxval{z}{i}{}$ is the (possibly noisy) measurement at $\idxpt{x}{i}{}$; we may refer to a sample just by $\idxpt{x}{i}{}$, as appropriate to the context. For any $\idxpt{x}{i}{}$, we define a displacement vector $\bm{a}_i = \bm{x}' - \idxpt{x}{i}{}$.

For brevity, we use the following shorthands for the value, gradient, and Hessian of $u$ at the query point: $u' = u(\bm{x}')$, $\bm{g}' = \nabla u(\bm{x}')$, $\bm{H}' = \nabla^2 u(\bm{x}')$; accordingly, we use $\idxval{u}{i}{}$, $\idxpt{g}{i}{}$, and $\idxpt{H}{i}{}$ to refer to $u(\idxpt{x}{i}{})$, $\nabla u(\idxpt{x}{i}{})$, and $\nabla^2 u(\idxpt{x}{i}{})$, respectively.

The Hessian $\bm{H}'$ (resp. $\idxpt{H}{i}{}$) is used only in the intermediate Taylor-based relations of Section~\ref{sec:taylor-uq}; however, in most of our mathematical development, we use its symmetric half-vectorisation $\bm{h}' \doteq \mathrm{vech}(\bm{H}') \in \mathbb{R}^{D_h}$, $D_h \doteq \frac{D(D+1)}{2}$ (and likewise $\idxpt{h}{i}{} \doteq \mathrm{vech}(\idxpt{H}{i}{})$). We define the \textit{state triple} at the query point as the column vector
\[
    \bm{\theta}' \doteq \triple{u'}{\bm{g}'}{\bm{h}'} \in \mathbb{R}^{1+D+D_h}.
\]
\noindent For a sample $\idxpt{x}{i}{}$, the corresponding triple is $\idxpt{\theta}{i}{} \doteq \triple{\idxval{u}{i}{}}{\idxpt{g}{i}{}}{\idxpt{h}{i}{}}$.

We reserve an undecorated $\bm\theta$ (with numeric subscripts $\bm\theta_1, \bm\theta_2, \ldots$ where more than one is needed) to denote a \textit{generic, arbitrary} element of an admissible set, e.g.\ inside a $\sup$, $\inf$, or $\forall$ quantifier. On the other hand, $\tilde{\cdot}$ denotes a specific, \textit{named} admissible triple, e.g.\ $\tilde{\bm{\theta}}' \doteq \triple{\tilde{u}'}{\tilde{\bm{g}}'}{\tilde{\bm{h}}'}$ for a nominal admissible triple used internally by the method described in this paper. All norms $\|\cdot\|$ are 2-norms, unless specified by a subscript. Lastly, $\underline{\sigma}(\cdot)$ denotes the minimum singular value of a matrix.

\section{Theoretical Framework}
\label{sec:theoretical-framework}

\subsection{Preliminaries}
We consider the estimation of an \textit{a priori} unknown state function $u(\bm{x})$, $u: \mathcal{R} \rightarrow \mathbb{R}$, where $\mathcal{R} \subset \mathbb{R}^D$ is a closed and compact $D$-dimensional convex region. While we do not have access to any closed-form expression of $u(\bm{x})$, i.e., it is ``black-box'', we take an assumption on its continuity, which we state as follows:

\begin{assumption}[Lipschitz Hessian continuity]
\label{assm:lipschitz-hessian}
    The latent value function $u(\bm{x})$ has a Lipschitz-continuous Hessian with an \textit{a priori} known constant $L_H > 0$,
    \[
        \left\|\nabla^2u(\bm{x}_1) - \nabla^2u(\bm{x}_2)\right\|_\infty \leq L_H \|\bm{x}_1 - \bm{x}_2\|, \;\; \bm{x}_1, \bm{x}_2 \in \mathcal{R}.
    \]
\end{assumption}

\noindent Furthermore, we can access individual values at discrete sampling points $\idxpt{x}{i}{} \in \mathcal{R}, \; i \in \mathbb{N}$ through sampling, for which we state another assumption:

\begin{assumption}[Sampling with finite noise bound]
\label{assm:finite-noise}
    For a discrete point $\idxpt{x}{i}{}$, it is possible to acquire the corresponding function value $u(\idxpt{x}{i}{})$, subject to noise with \textit{a priori} known finite bound $\bar{\epsilon} \in [0, \infty)$:

    \[
        \idxval{z}{i}{} = u(\idxpt{x}{i}{}) + \epsilon_i, \;\; |\epsilon_i| \leq \bar{\epsilon}.
    \]
\end{assumption}

\noindent The above assumption is commonly taken in practical settings, where evaluations are subject to (additive) sensor or process noise. However, we also cover the noiseless case: Assumption~\ref{assm:finite-noise} only assumes that the noise bound is \textit{a priori} known. Note that we only have access to the function values, and not to any first- nor higher-order derivatives, i.e., we are only given zeroth-order information.

Consider a data set $\itrpt{X}{n}{}$ composed of $n$ (possibly noise-injected) samples
\[
    \itrpt{X}{n}{} \doteq \left\{(\idxpt{x}{1}{}, \idxval{z}{1}{}), \ldots, (\idxpt{x}{n}{}, \idxval{z}{n}{})\right\}.
\]
\noindent To tackle the estimate of state $u'$ at an unsampled point (henceforth referred to as the ``query point'') $\bm{x}' \in \mathcal{R}$, we use our \textit{a priori} information as provided by the existing samples $\itrpt{X}{n}{}$ and Assumptions \ref{assm:lipschitz-hessian}-\ref{assm:finite-noise}. We now formally state our problem:

\begin{problem}[State triple estimation]
\label{prob:estimation-state-triple}
    Let $\itrpt{X}{n}{}$ be a data set composed of state value samples satisfying Assumptions~\ref{assm:lipschitz-hessian}-\ref{assm:finite-noise}. Design a method that, for an interior query point $\bm{x}' \in \mathcal{R}$, reports a state triple estimate $\tilde{\bm{\theta}}'$, and its corresponding bounding set $\admtriplesset$ satisfying the finite-sample coverage on the true triple,
    \[
        \bm{\theta}' \in \admtriplesset.
    \]
    In the noiseless case, both $\tilde{\bm{\theta}}'$ and $\admtriplesset$ shall converge to the true triple $\bm{\theta}'$ as sampling becomes dense in the sense defined later in Definition~\ref{def:dense-sampling}.
\end{problem}

We introduce in this paper a simple approach to solve Problem~\ref{prob:estimation-state-triple}, using a combination of Taylor series expansion and set membership approaches.

\subsection{Taylor-based uncertainty quantification and estimation}
\label{sec:taylor-uq}

We express $u'$ with respect to (w.r.t.) an existing sample $\idxval{u}{i}{}$ and its (unknown) derivatives up to the second order, with a truncation error $\delta_i$
\begin{equation}
\label{eqn:value-prediction}
    u' = \idxval{u}{i}{} + \idxptT{g}{i}{}\bm{a}_i + \frac{1}{2} \bm{a}_i^\top \idxpt{H}{i}{} \bm{a}_i + \delta_i,\;\; |\delta_i| \leq \frac{L_H}{6} \|\bm{a}_i\|^3.
\end{equation}
Invoking kinematic relations of the gradients and Hessians at $\bm{x}'$ and $\idxpt{x}{i}{}$,
\begin{align}
    \bm{g}' &= \idxpt{g}{i}{} + \frac{1}{2}\left(\bm{H}' + \idxpt{H}{i}{}\right)\bm{a}_i + \bm{\omega}_i,\;\; \|\bm{\omega}_i\|_\infty \leq \frac{L_H}{4} \|\bm{a}_i\|^2 \label{eqn:gradient-prediction}\\
    \bm{H}' &= \idxpt{H}{i}{} + \bm{\Gamma}_i, \;\; \|\bm{\Gamma}_i\|_\infty \leq L_H\|\bm{a}_i\|. \label{eqn:hessian-prediction}
\end{align}

Using $\bm{h}' = \mathrm{vech}(\bm{H}')$ and with $\bm{D}_D$ as the duplication matrix, we denote $\bm{\lambda}_i = \frac{1}{2} \bm{D}_D^\top(\bm{a}_i \otimes \bm{a}_i)$ such that (s.t.) $\frac{1}{2} \bm{a}_i^\top \idxpt{H}{i}{} \bm{a}_i = \bm{\lambda}_i^\top \idxpt{h}{i}{}$, and $\bm{c}_i = \frac{1}{2} \bm{D}_D^\top (\bm{a}_i \otimes \bm{I}_D)$ s.t. $\frac{1}{2} \idxpt{H}{i}{} \bm{a}_i = \bm{c}_i^\top \idxpt{h}{i}{}$. \eqref{eqn:value-prediction}-\eqref{eqn:hessian-prediction} then become
\begin{align}
    u' &= \idxval{u}{i}{} + \bm{a}_i^\top \idxpt{g}{i}{} + \bm{\lambda}_i^\top \idxpt{h}{i}{} + \delta_i,\;\; |\delta_i| \leq \frac{L_H}{6} \|\bm{a}_i\|^3 \\
    \bm{g}' &= \idxpt{g}{i}{} + \bm{c}_i^\top(\bm{h}' + \idxpt{h}{i}{}) + \bm{\omega}_i,\;\; \|\bm{\omega}_i\|_\infty \leq \frac{L_H}{4} \|\bm{a}_i\|^2 \\
    \bm{h}' &= \idxpt{h}{i}{} + \bm{\gamma}_i, \;\; \|\bm{\gamma}_i\|_\infty \leq L_H\|\bm{a}_i\|.
\end{align}

We consider the neighbouring samples set w.r.t. $\bm{x}'$, denoted as $\mathcal{N}^{\iter{n}} \subset \itrpt{X}{n}{}$ with cardinality $N$.
Adding the backward value expansions\footnote{The backwards gradient and Hessian relations only duplicate the original ones.}, and recasting all the previous equations as inequalities, we arrive at:
\begin{subequations}
\label{eqn:taylor-noise-relationships}
\begin{align}
    \left|u' - \left( \idxval{u}{i}{} + \bm{a}_i^\top \idxpt{g}{i}{} + \bm{\lambda}_i^\top \idxpt{h}{i}{} \right)\right| &\leq \frac{L_H}{6} \|\bm{a}_i\|^3 \\
    \left|\idxval{u}{i}{} - \left( u' - \bm{a}_i^\top \bm{g}' + \bm{\lambda}_i^\top \bm{h}' \right)\right| &\leq \frac{L_H}{6} \|\bm{a}_i\|^3 \label{eqn:taylor-from-query-point}\\
    \left\|\bm{g}' - \left(\idxpt{g}{i}{} + \bm{c}_i^\top(\bm{h}' + \idxpt{h}{i}{}) \right)\right\|_\infty &\leq \frac{L_H}{4} \|\bm{a}_i\|^2 \\
    \left\|\bm{h}' - \idxpt{h}{i}{}\right\|_\infty &\leq L_H\|\bm{a}_i\| \\
    |\idxval{u}{i}{} - \idxval{z}{i}{}| &\leq \bar{\epsilon}.
\end{align}
\end{subequations}
\noindent We collect the unknown state values and derivatives from the inequalities above, and organise as triples at the query point $\bm{x}'$ and its $N$ neighbours
\begin{equation}
\label{eqn:diona-decision-variable}
    \bm{\psi} \doteq \mathrm{col}\!\left(\bm{\theta}', \idxpt{\theta}{i_1}{}, \idxpt{\theta}{i_2}{}, \ldots, \idxpt{\theta}{i_{N}}{}\right),
\end{equation}
\noindent and the inequalities can be expressed in the form of
\[
    \bm{\Xi} \bm{\psi} \leq \bm{b}.
\]
\noindent with the matrices $\bm{\Xi}$ and $\bm{b}$ composed of terms pertaining to the displacements $\bm{a}_i$ and the Hessian Lipschitz constant $L_H$. Since the samples are fixed, $\bm{a}_i$ are constants and the inequality above is linear w.r.t. the unknowns, despite describing quadratic Taylor relations. Let us define the admissible set of state triples
\begin{equation}
\label{eqn:admissible-set-triples}
    \mathcal{A} = \left\{\bm{\psi}: \bm{\Xi}\bm{\psi} \leq \bm{b} \right\}. 
\end{equation}
Since $\bm{\psi}$ is the collection of the admissible triples at the query point and the neighbouring samples, we finally have
\[
    \mathcal{T}'{}^{\langle n \rangle} = \text{proj}_{\bm{\theta}'} \mathcal{A}, 
\]
\noindent and the upper- and lower bounds for $u(\bm{x}')$ are
\begin{subequations}
\label{eqn:upper-lower-bound-data-only}
\begin{align}
    \overline{u}' &= \max_{\bm{\theta} \in \admtriplesset} u \\
    \underline{u}' &= \min_{\bm{\theta} \in \admtriplesset} u.
\end{align}
\end{subequations}
\noindent Consequently, the state value estimate is the midpoint
\begin{equation}
    \tilde{u}' = \frac{\overline{u}' + \underline{u}'}{2},
\end{equation}
\noindent and the uncertainty measure is the difference
\begin{equation}
    w = \overline{u}' - \underline{u}'.
\end{equation}
\noindent Similar bounds and central estimates on $\bm{g}'$ and $\bm{h}'$ are readily available from $\admtriplesset$, using the same approach. Fig.~\ref{fig:diona-visualisation} shows a visual of curvature-aware uncertainty bounds on a 1-dimensional function given noiseless samples and an \textit{a priori} known $L_H$, compared to nonlinear set membership (NSM)-based bounds with a known Lipschitz constant $\gamma$.  

\begin{figure}[t]
\centering
\includegraphics[width=\textwidth]{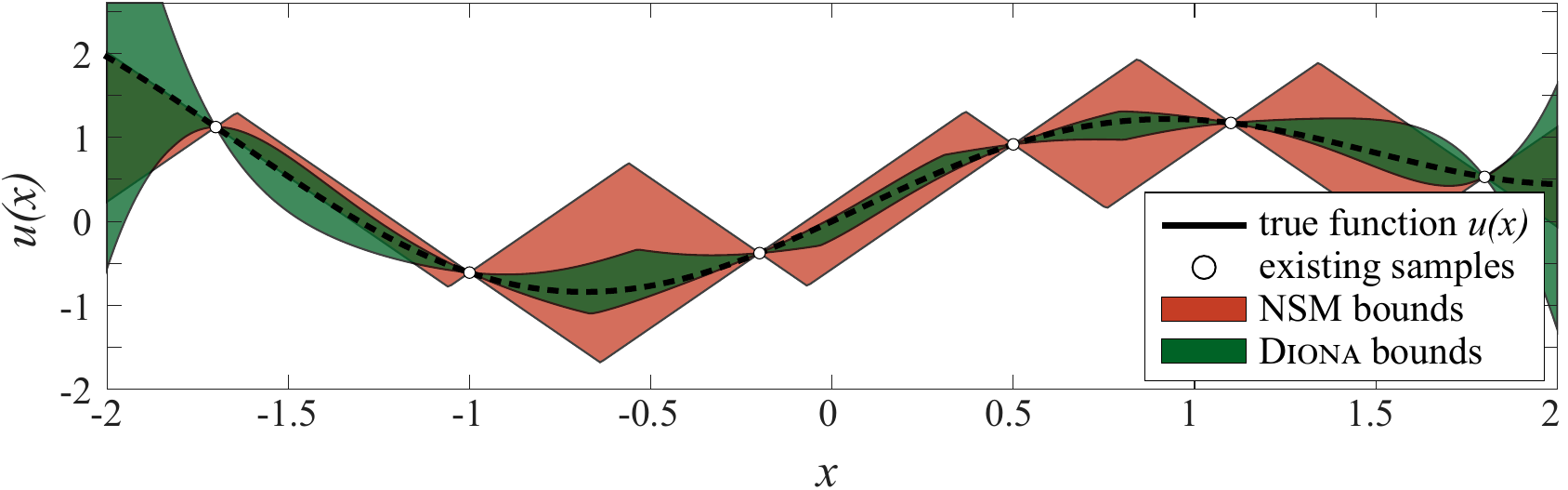}
\caption{1-dimensional visualisation of Hessian Lipschitz-based \textsc{Diona} uncertainty \\bounds from noiseless samples, with $u(x)=\sin(2x)+0.3x^2$, in comparison with \\a value Lipschitz-based NSM bounds.}
\label{fig:diona-visualisation}
\end{figure}

\subsection{Neighbour samples}
\label{sec:neighbour-samples}

For a candidate sample $\idxpt{x}{i}{}$ at trial radius $r$, we denote $\hat{\bm{a}}_i \doteq \frac{\bm{a}_i}{r}$ and $\hat{\bm{\lambda}}_i \doteq \frac{\bm{\lambda}_i}{r^2}$, and build a row $\bm{\phi}_i \doteq [1 ~ -\hat{\bm{a}}_i^\top ~ \hat{\bm\lambda}_i^\top]^\top$, following the Taylor expansion at \eqref{eqn:taylor-from-query-point}. For a candidate set $\mathcal{N}$, let $\bm{\Phi}$ be the stacked rows $\bm{\phi}_i^\top$, $\idxpt{x}{i}{}\in\mathcal{N}$.

\begin{definition}[$\mu_0$-poised neighbour set]
\label{def:poised-neighbour-set}
    A neighbour set $\mathcal{N}$ with $N \;\geq 1+D+D_h$ is $\mu_0$-\textit{poised} if for a given $\mu_0 > 0, \; \underline{\sigma}(\bm{\Phi}) \geq \mu_0 \sqrt{N}$.
\end{definition}

\begin{assumption}
\label{assm:poised-neighbour-set}
    For a query point $\bm{x}'$ in the interior of $\mathcal{R}$, the data set $\itrpt{X}{n}{}$ (with $n$ large enough) and $0 < \mu_0 < \bar{\mu}_D$\footnote{$\bar{\mu}_D$ is the supremum of $\underline{\sigma}(\bm{\Phi})/\sqrt{N}$ over all point sets, i.e.\ the square root of the $E$-optimal design value for second-order regression on the unit ball, obtained as in \cite{dette2014eoptimal} with the scaling of $\bm{\Phi}$: $\bar{\mu}_D = (16D+1)^{-1/2}$ for $D \leq 4$ and $\bar{\mu}_D = ((D+4)^2+1)^{-1/2}$ for $D \geq 4$.}, we have access to its corresponding minimum-radius $\mu_0$-poised neighbour set
    \[
        \mathcal{N}^{\iter{n}*} = \argmin_{\substack{\mathcal{M} \subset \itrpt{X}{n}{}, \\ \textrm{$\mu_0$-poised}}} \max_{\idxpt{x}{i}{} \in \mathcal{M}} \| \bm{x}' - \idxpt{x}{i}{} \|
    \]
    with radius $r^{\iter{n}*}_{\mathcal{N}}$.
\end{assumption}

In the succeeding theoretical discussions, we consider an interior query point $\bm{x}'$ with a corresponding poised neighbour set $\mathcal{N}$, indicating the richness of information to estimate $\bm{g}'$ and $\bm{H}'$.

\section{Theoretical properties}
\label{sec:theoretical-properties}

Here we introduce theoretical properties of our proposed method, whose proofs are all delegated to the Appendix.

\begin{lemma}[Containment]
\label{lemma:containment}
    Given the query point $\bm{x}'$, the data set $\itrpt{X}{n}{}$, and under Assumptions~\ref{assm:lipschitz-hessian}-\ref{assm:finite-noise}, the true query triple $\bm{\theta}'$ is contained in $\admtriplesset$.
\end{lemma}

\textit{Proof.} See Appendix.

With Lemma~\ref{lemma:worst-case-diameter-states-derivatives}, the true triples at the query point $\bm{\theta}'$ and its neighbours are admissible for any number of samples by construction. This result therefore relies on the validity or \textit{a priori} knowledge of $L_H$ and $\bar\epsilon$: resorting to their (under)estimates (as in Section~\ref{sec:parameter-estimation}) voids this guarantee.

\begin{lemma}[Worst-case diameter of admissible states and derivatives]
\label{lemma:worst-case-diameter-states-derivatives}
    For the query point $\bm{x}'$ and data set $\itrpt{X}{n}{}$, consider two admissible triples $\bm{\theta}_1, \bm{\theta}_2 \in \admtriplesset$, and denote the differences in their latent states and derivatives as $\Delta u = u_1 - u_2, \; \Delta \bm{g} = \bm{g}_1 - \bm{g}_2, \; \Delta \bm{h} = \bm{h}_1 - \bm{h}_2$. Then, under Assumptions~\ref{assm:lipschitz-hessian}-\ref{assm:finite-noise}, and with any $\mu_0$-poised neighbour set $\mathcal{N}$,
    \begin{subequations}
    \label{eqn:worst-case-diam}
    \begin{align}
        | \Delta u | &\leq d_u = \frac{1}{\mu_0} \left(2\bar{\epsilon} + \frac{L_H}{3}r^3_{\mathcal{N}}\right) \\
        \| \Delta \bm{g} \|_2 &\leq d_{\bm{g}} = \frac{1}{\mu_0} \left(\frac{2\bar{\epsilon}}{r_{\mathcal{N}}} + \frac{L_H}{3}r^2_{\mathcal{N}}\right) \\
        \| \Delta \bm{h} \|_2 &\leq d_{\bm{h}} = \frac{1}{\mu_0} \left(\frac{2\bar{\epsilon}}{r^2_{\mathcal{N}}} + \frac{L_H}{3}r_{\mathcal{N}}\right)
    \end{align}
    \end{subequations}    
    where $r_{\mathcal{N}} = \max_{\idxpt{x}{i}{} \in \mathcal{N}} \|\idxpt{x}{i}{} - \bm{x}'\|$.
\end{lemma}

\textit{Proof.} See Appendix.

In the noiseless setup, going from $u'$ to $\bm{g}'$ to $\bm{h}'$ costs the bounds by one power of the neighbourhood radius ($r_{\mathcal{N}}^3$, to $r_{\mathcal{N}}^2$ and $r_{\mathcal{N}}$): considering a neighbourhood with half the radius tightens the $u'$ bound eightfold. With noise, $d_u$ tends to a floor of $\frac{2\overline{\epsilon}}{\mu_0}$ as $r_{\mathcal{N}}\rightarrow0$. In contrast, $d_{\bm{g}}$ and $d_{\bm{h}}$ grow again when considering neighbourhoods with $r_{\mathcal{N}} < \left(\frac{3\overline{\epsilon}}{L_H}\right)^{1/3}$ and $r_{\mathcal{N}} < \left(\frac{12\overline{\epsilon}}{L_H}\right)^{1/3}$, respectively.


\begin{remark}[Information radius]
\label{remark:information-radius}
    When the state triple estimate $\tilde{\bm{\theta}}'$ is taken as the centroid\footnote{which is not required to be in the feasible triples set, see \cite{milanese2004setmembership}} of the componentwise minimum bounding box of $\admtriplesset$, then the $\infty$-norm information radii (as in \cite{milanese2004setmembership}) on $u$, $\bm{g}$, and $\bm{h}$ are bounded by $\frac{1}{2}d_u, \frac{1}{2}d_{\bm{g}},$ and $\frac{1}{2}d_{\bm{h}}$, respectively, since $\|\cdot\|_\infty \leq \|\cdot\|_2$.
\end{remark}

\begin{definition}[Dense sampling]
\label{def:dense-sampling}
Let $\itrpt{X}{n}{}$ denote the set of sampled points after $n$ samplings. We say that a sampling regime is dense in $\mathcal{R}$ if
\[
    \max_{\bm{x}' \in \mathcal{R}} \min_{\idxpt{x}{i}{} \in \itrpt{X}{n}{}} \left\|\bm{x}' - \idxpt{x}{i}{} \right\| \xrightarrow[n\rightarrow \infty]{} 0.
\]
\end{definition}

\begin{lemma}
\label{lemma:dense-samples}
    Given a query point $\bm{x}' \in \mathcal{R}$, and under Assumptions~\ref{assm:lipschitz-hessian}-\ref{assm:poised-neighbour-set} and dense sampling, $r_{\mathcal{N}}^{\iter{n}*} \xrightarrow[n \rightarrow \infty]{} 0$.
\end{lemma}

\textit{Proof.} See Appendix.

\begin{theorem}[Convergence to the true triple]
\label{theorem:convergence-true-triple}
    Under Assumptions \ref{assm:lipschitz-hessian}-\ref{assm:poised-neighbour-set} and $\bar{\epsilon}=0$, the admissible set of triples at $\bm{x}'$ shrinks to the true one, i.e.,
    \begin{equation}
    \label{eqn:convergence-true-triple}
        \sup_{\bm{\theta} \in \admtriplesset} \| \bm{\theta} - \bm{\theta}' \| \xrightarrow[n\rightarrow \infty]{} 0
    \end{equation}
    as the samples set densely covers the search space.
\end{theorem}

\textit{Proof.} See Appendix.

The result in Theorem~\ref{theorem:convergence-true-triple} relies on $\overline{\epsilon}=0$, a known $L_H$, and a poised neighbour set with radius that shrinks with sampling (Lemma~\ref{lemma:dense-samples}). 
Given Theorem~\ref{theorem:convergence-true-triple} and Remark~\ref{remark:information-radius}, the estimated triple converges to the true one, $\|\tilde{\bm{\theta}}' - \bm{\theta}'\|\rightarrow 0$.


\section{Implementation notes}
\label{sec:implementation-notes}

We tackle here practical workarounds for settings where we do not satisfy Assumptions~\ref{assm:lipschitz-hessian}-\ref{assm:poised-neighbour-set}. We use a combination of samples-based estimation of $\overline{\epsilon}, L_H$ and a greedy poised neighbours selection approach.

\subsection{Parameter estimation}
\label{sec:parameter-estimation}

In many applications, the Hessian Lipschitz constant $L_H$, and the noise bound $\bar\epsilon$ are not \textit{a priori} known, and must be estimated from the data set $\itrpt{X}{n}{}$. To estimate $\bar \epsilon$, any preferred method can be used, and here we use the set membership-based approach from \cite{novara2016setmembership}, to which we refer the reader for more information.

We estimate $L_H$ by solving an LP over the active points set $\mathcal{P} = \bm{x}' \cup \mathcal{N}$,
\[
    \tilde{L}_H = \min_{L_H,\, \bm\psi_{\mathcal{P}}} L_H \quad \text{s.t.} \quad \bm\Xi_{\mathcal{P}}\bm\psi_{\mathcal{P}} \leq \bm{b}_{\mathcal{P}},
\]
where $\bm\Xi_{\mathcal{P}}, \bm{b}_{\mathcal{P}}$ are assembled similarly as in Section~\ref{sec:taylor-uq}, but with $L_H$ rearranged/cast as a decision variable (stacked with $\bm{\theta}', \idxpt{\theta}{i_1}{}, \dots$, see \eqref{eqn:diona-decision-variable}).

Resorting to $L_H$ and $\bar\epsilon$ estimates invalidates the formal theoretical guarantees. Nevertheless, we show in the results that \textsc{Diona} still performs competitively even with estimated constants.

\subsection{Practical neighbour selection}

Selecting a minimum-radius $\mu_0$-poised neighbour set to satisfy Assumption~\ref{assm:poised-neighbour-set} is combinatorial w.r.t. the samples, and can be computationally prohibitive. Therefore, we have built a simple greedy method to collect the neighbours set. This involves starting from a nearest samples set as initial candidates for selection, and building a neighbours set achieving a full-rank feature matrix $\bm{\Phi}_{\mathcal{C}}$. From this neighbours set, we then iteratively add further samples to achieve the $\mu_0$-poisedness.

Upon empirical testing, we opted for three passes of Algorithm~\ref{alg:neighbour-selection} (each successive pass excludes previously selected samples), and use the union of all such neighbour samples to provide redundancy for computing $\tilde{L}_H$ and $\underline{u}'$, and $\overline{u}'$. This is to achieve a more informative $\tilde{L}_H$ estimate\footnote{This is because having only one neighbour ``layer'' will almost always fit a quadratic, corresponding to $\tilde{L}_H \approx 0$.} and more reliable bounds especially under noise. With \textit{a priori} known $\overline{\epsilon}$ and $L_H$,  Lemma~\ref{lemma:worst-case-diameter-states-derivatives} still holds for the first neighbour set as the succeeding neighbour ``layers'' only serve to add more constraints.


\begin{algorithm}[t]
    \small
    \DontPrintSemicolon
    \caption{Practical $\mu_0$-poised neighbour selection}
    \label{alg:neighbour-selection}
    \KwInput{query point $\bm{x}'$, poisedness threshold $\mu_0>0$, maximum~selected~neighbours $\overline{N}$}
    \KwData{samples set $\itrpt{X}{n}{}$}
    \KwOutput{$\tilde{\mathcal N}$, poisedness $\mu(\tilde{\mathcal N})$, radius $r_{\mathcal N}$}
    $m\leftarrow1+D+\frac{D(D+1)}{2}$; fetch at least $m$ closest samples to $\bm{x}'$\;
    \While{true}{
        Retrieve the current nearest-candidate pool $\mathcal C$\;
        Build $[1,-\hat{\bm a}_i^\top,\hat{\bm\lambda}_i^\top]^\top$ for $\idxpt{x}{i}{}\in\mathcal C$, build feature matrix $\bm{\Phi}_{\mathcal{C}}$\;
        Build proposed full-rank $m$-point set $\tilde{\mathcal N}$ with column-pivoted QR on $\bm{\Phi}_{\mathcal{C}}^\top$, expanding $\mathcal C$ as necessary\;
        \If{$\tilde{\mathcal N}$ rank-deficient \textbf{and} all samples exhausted}{
            \Return{\textsc{NeighbourFailure}}\;
        }
        \While{$|\tilde{\mathcal N}| \leq \overline{N}$}{
            $r_{\mathcal N}\leftarrow\max_{i\in\tilde{\mathcal N}}\|\idxpt{x}{i}{}-\bm x'\|$\;
            Rebuild $\bm\Phi_{\mathcal N}$, normalised using radius $r_{\tilde{\mathcal N}}$\;
            Using SVD, $\mu(\tilde{\mathcal N})\leftarrow\underline\sigma(\bm\Phi_{\mathcal N})/\sqrt{|\tilde{\mathcal N}|}$\;
            \If{$\mu(\tilde{\mathcal N})\geq\mu_0$}{\Return{$\tilde{\mathcal N},\mu(\tilde{\mathcal N}),r_{\tilde{\mathcal N}}$}\;}
            \Else{
                \If{$|\tilde{\mathcal N}| = \overline{N}$ \textbf{or} all points in $\mathcal C$ exhausted}{
                    \textbf{break}\;
                }
                \Else{
                    Add $\idxpt{x}{i}{} \in \mathcal{C} \setminus \tilde{\mathcal{N}}$ that results in maximum $\mu(\tilde{\mathcal{N}} \cup \idxpt{x}{i}{})$\;
                }
            }
            
        }
        \If{all samples exhausted}{
            \Return{\textsc{NeighbourFailure}}\;
        }
        \Else{
            Expand candidate pool $\mathcal{C}$\;
        }
    }
\end{algorithm}

\section{Direct Inference on Admissible Sets (Diona)}
\label{sec:diona-method}

We build on our theoretical results and use our practical workarounds to build a practical nonparametric regression method called Direct Inference on Admissible Sets (\textsc{Diona}), summarised in Algorithm~\ref{alg:diona}. \textsc{Diona} picks neighbouring samples of the query point $\bm{x}'$, and assembles the joint linear constraints on their corresponding triples (collected in the decision variable $\bm{\psi}$), leading to an optimisation problem for local regression. This replaces a long and global model training process, with a per-query optimisation problem to compute the state estimate and its corresponding uncertainty.


\begin{algorithm}[t]
    \small
    \DontPrintSemicolon
    \caption{Direct Inference on Admissible Sets (\textsc{Diona})}
    \label{alg:diona}
    \KwInput{query $\bm{x}'$, samples $\itrpt{X}{n}{}$, poisedness threshold $\mu_0>0$}
    \KwOutput{scalar estimate $\tilde u'$, bounds $[\underline u',\overline u']$}
    Extract noise estimate $\tilde{\epsilon}$ (if true noise bound $\overline{\epsilon}$ not \textit{a priori} known, see \cite{novara2016setmembership})\;
    $\mathcal{N} \leftarrow \textsc{NeighbourSelection}(\bm{x}',\itrpt{X}{n}{},\mu_0)$ (Algorithm~\ref{alg:neighbour-selection})\;
    Solve for (local) Hessian Lipschitz constant estimate $\tilde{L}_H$ (Section~\ref{sec:parameter-estimation})\;
    Assemble the admissible triples polytope $\mathcal{A}=\{\bm\psi:\bm\Xi\bm\psi\leq\bm b\}$ over $\{\bm{x}'\}\cup\mathcal N$\;
    Solve for admissible state bounds $\underline u'\leftarrow\min_{\bm\psi\in\mathcal A}u'$; $\overline u'\leftarrow\max_{\bm\psi\in\mathcal A}u'$\;
    \Return{$\tilde{u}' = \frac{\underline u'+\overline u'}{2}$, $[\underline u',\overline u']$\;}
\end{algorithm}

\section{Comparative study}

We subject \textsc{Diona} to an extensive benchmark test for data-driven regression of benchmark functions, comparing with methods from literature and common practice in terms of prediction accuracy, uncertainty quantification, and computational burden. The Julia code base, as well as the \textsc{Diona} package, is available in \verb|https://github.com/lorenzosabugjr/Diona|.

\subsection{Benchmark setup}

We consider the regression of different underlying functions (from $D=3$ to $D=8$) summarised in Table~\ref{tab:benchmark-functions}, taken from \cite{surjanovic2013virtual}. To present a comprehensive analysis, we have considered different data layouts, and with noisy samples. For each test function $u(\bm{x})$, we generated data sets of increasing density (with $500D$, $1000D$, $2000D$ samples) in a Sobol sequence layout on the respective domains. To evaluate the practical noise robustness, we tested with exact (noiseless) samples and with additive Gaussian noise (20 dB SNR)--a deviation from Assumption~\ref{assm:finite-noise} for \textsc{Diona}.

We compare all methods in terms of the (approximate) $L_2$ relative error which, with $\bm{U} \doteq \Big(u(\bm{q})\Big)_{\bm{q} \in \bm{Q}}$ and $\tilde{\bm{U}} \doteq \Big(\tilde{u}(\bm{q})\Big)_{\bm{q} \in \bm{Q}}$, we measure as
\[
    \epsilon_{L2} = \frac{\left\|\bm{U} - \tilde{\bm{U}}\right\|}{\|\bm{U}\|}
\]
with $\bm{Q}$ being a collection of $5000D$ Sobol-distributed query points $\bm{q}$ (and ensuring no coincidence with any training samples). 

For regressors with UQ---GP, NSM, \textsc{Diona}---we measure the relative uncertainty width, given for every $\bm{q} \in \bm{Q}$ as
\[
    w_r(\bm{q}) = \frac{w(\bm{q})}{|u(\bm{q})|}
\]
noting the denominator being the \textit{truth} function value. Lastly, we compare among these UQ-capable regressors the aggregate \textit{coverage} of the truth value within the estimated bounds
\[
    \rho = \frac{\sum_{\bm{q}\in\bm{Q}} \mathds{1}\Big(\underline{u}(\bm{q}) \leq u(\bm{q}) \leq \overline{u}(\bm{q})\Big)}{|\bm{Q}|}
\]
with $\mathds{1}(\cdot)$ being the indicator function. As the stochastic uncertainty for GP is of a different nature than the (estimated) worst-case bounds of NSM and \textsc{Diona}, we measure the GP-specific uncertainty metrics considering its $\pm2\sigma_{GP}$ bounds. A low $w_r$ and high $\rho$ are essential for bound informativeness.

To compare the computational burden, we measured the (offline) training time (if applicable), and the online per-query (inference) times. All the computations are performed in a PC with Intel Core~i9 3.20~GHz, 160~GB RAM, and NVIDIA RTX~4500~Ada (24~GB VRAM). All the computations were run on the CPU, except for the NN GPU/CUDA-based pre-training.

\begin{table}[t]
\centering
\begin{tabular}{p{1.6in}p{2.5in}}
\toprule
\textbf{Name} & \textbf{Details} \\
\midrule
Neural network (NN) \cite{bishop1995neural} & Implemented with \verb|Lux.jl|. Fully connected network with $\ceil{D/2}+1$ hidden layers (64 tanh units). Trained to minimise MSE with Adam (25000 epochs with learning rate of $1\times 10^{-3}$). Used CUDA via \verb|LuxCUDA.jl|. \\
Gaussian process (GP) \cite{rasmussen2006gaussian}            & Sparse variational GP using \verb|AbstractGPs.jl|, \verb|ApproximateGPs.jl|, and \verb|KernelFunctions.jl|, Mat\'{e}rn-5/2 kernel, and $100D$ inducing points.                    \\
Nonlinear set membership (NSM) \cite{milanese2004setmembership}   & Implemented in Julia with estimated noise bound $\tilde{\epsilon}$ and value Lipschitz constant $\tilde{\gamma}$ computed as in \cite{novara2016setmembership}. $\tilde{\gamma}$ is then scaled by 2.5. \\
Radial basis functions (RBF) \cite{buhmann2003radial}     & Using \verb|Surrogates.jl| \textsc{RadialBasis} with the default linear radial kernel, with all default values.                    \\   
Inverse distance weighting (IDW) & Using \verb|Surrogates.jl| \textsc{InverseDistanceSurrogate}, using all training samples. Used inverse-distance exponent $p=D+2$, otherwise all default settings. \\
\textsc{Diona} (ours)            & As described in Section~\ref{sec:implementation-notes}-\ref{sec:diona-method}, and with $\tilde{L}_H$ scaled by 2.5. LPs are computed with \verb|HiGHS.jl|.                    \\
\bottomrule
\end{tabular}
\caption{Compared data-driven regressors}
\label{tab:benchmark-regressors}
\end{table}

\begin{table}[t]
\centering
\begin{tabular}{lllp{2.0in}}
\toprule
\textbf{Function name} & \textbf{$D$} & \textbf{Domain} & \textbf{Properties} \\
\midrule
Ishigami               & 3           & $[-\pi,\pi]^3$ & Oscillatory, with growing amplitude, variable interactions                    \\
Park                   & 4           & $[0,1]^4$ & Nonseparable, spatially varying curvature                    \\
Friedman               & 5           & $[0,1]^5$ & Nonlinear, nonseparable, with linear and quadratic additive terms \\
Hartmann               & 6           & $[0,1]^6$ & Multimodal                    \\
Piston (simulation)    & 7           & (see \cite{surjanovic2013virtual}) & Nonlinear physical model, input variable coupling                    \\
Rosenbrock             & 8           & $[-5,10]^8$ & Narrow curved valley, strong input coupling                   \\
\bottomrule
\end{tabular}
\caption{Benchmark functions}
\label{tab:benchmark-functions}
\end{table}

\subsection{Results: regression accuracy}
We show in Fig.~\ref{fig:results-relative-l2} the $L_2$ relative errors ($\epsilon_{L2}$) per benchmark function, in relation to the samples density and presence of (20~dB) noise. In the graphs, the different bars in the same bundle represent the metrics at increasing samples density. We see that for all noiseless cases, \textsc{Diona}, NSM, RBF, and IDW were consistent in their decreasing $\epsilon_{L2}$ trends with increasing samples density. The trends with \textsc{Diona} generally reflects our theoretical results (although we used only estimated $\overline{\epsilon}$ and $L_H$), that the admissible triples set (and in turn, the state bounds $\overline{u}, \underline{u}$) shrink with denser samples sets. In fact, the same arguments can be made for NSM, IDW, and RBF, because the anchor samples from which to make the predictions get closer and closer. On the other hand, there is no clear trend for NN $\epsilon_{L2}$ w.r.t. the samples set density. This is due to its nonconvex training process, which can converge to different weights, which, across the denser input samples, can be uncorrelated. For GP, the $L_2$ errors have stayed mostly the same, in the range of $7\times10^{-3}$ to $8\times10^{-3}$.

The trends and sensitivity of $\epsilon_{L2}$ with 20~dB noise paint a different picture, with '$+$' markers showing the metrics in the noisy case, and their gap from the bars showing the sensitivity. All the compared methods attained $\epsilon_{L2}$ ranging from $1\times10^{-2}$ to $1\times10^{-1}$, which is now a much narrower difference in accuracy results than in the noiseless case. This said, all the advantages of \textsc{Diona} for $\overline{\epsilon}=0$ are now erased: consistent with Lemma~\ref{lemma:worst-case-diameter-states-derivatives}, $\overline{\epsilon} >0$ no longer guarantees that denser data sets shrink the worst-case error to zero. Nevertheless, its results are still highly competitive w.r.t. the other regressors.

\begin{figure}[t]
\centering
\includegraphics[width=\textwidth]{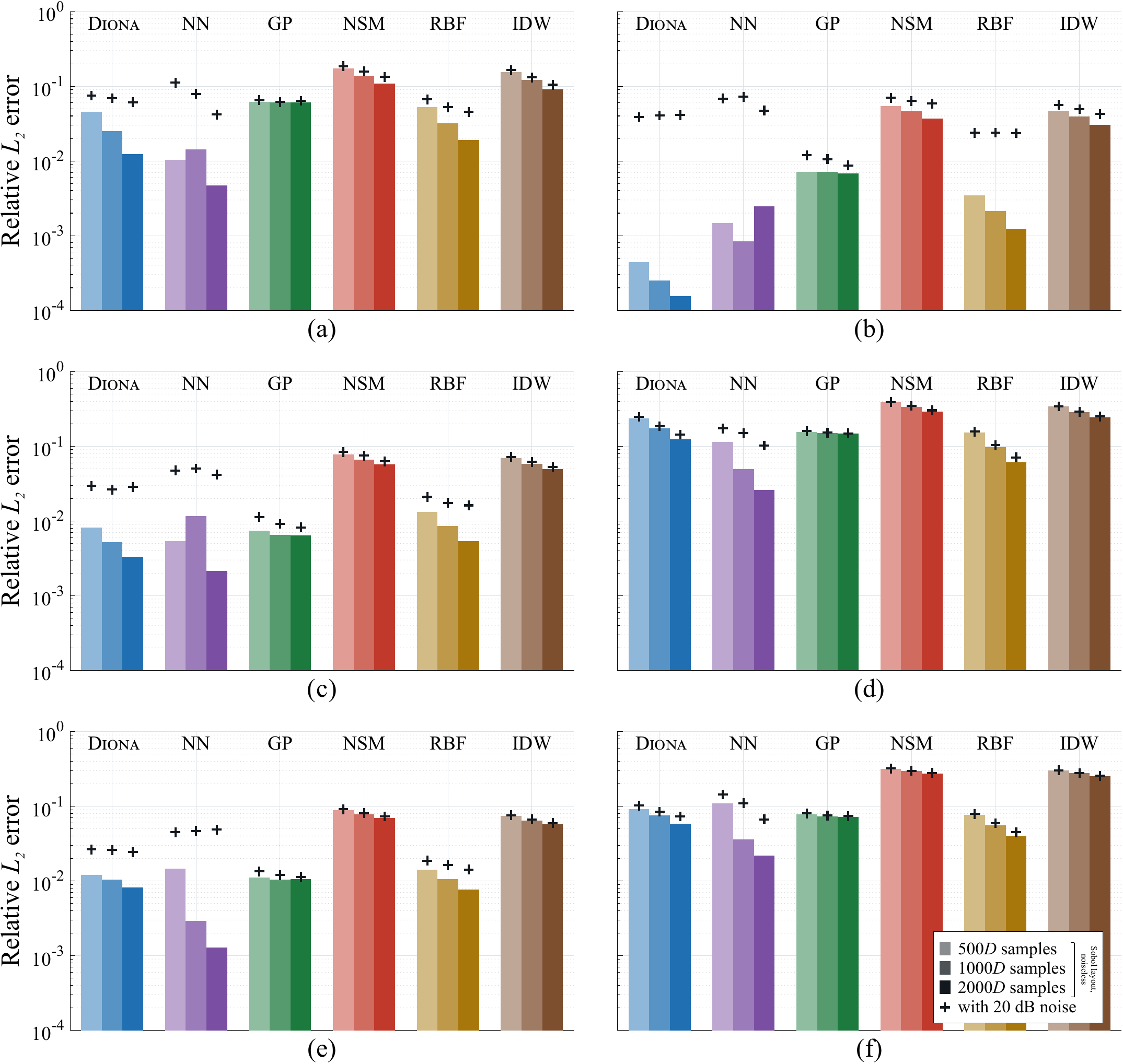}
\caption{Results: relative $L_2$ error\\(a) Ishigami, (b) Park, (c) Friedman, (d) Hartmann, (e) Piston, (f) Rosenbrock.}
\label{fig:results-relative-l2}
\end{figure}

\subsubsection{Results: uncertainty informativeness}
Fig.~\ref{fig:results-uncertainty} shows a $xy$-scatterplot with the median relative uncertainty width $w_r$ vs the truth coverage $\rho$. We plot here the resulting metrics for the \textsc{Diona}, GP, and NSM in increasing samples count/density (shows as larger markers), with noiseless and 20~dB noise-corrupted data sets shown as $\circ$ and $+$ markers, respectively. Markers to the lower-left corner of the plots are generally better (more informative UQ), signifying tighter intervals that include the truth values. Across all tested functions, NSM resulted in the highest median $w_r$, which in higher dimensional Hartmann ($D=6$) and Rosenbrock ($D=8$) even exceed $w_r = 1.0$ while having 100\% coverage, suggesting conservative and less informative bounds. This is attributed to its linear uncertainty growth w.r.t. the distance from the closest sample (see Fig.~\ref{fig:diona-visualisation}). Our \textsc{Diona}, on the other hand, has consistently achieved better $w_r$ than NSM, due to its consideration of the (estimated local) curvature. In fact, \textsc{Diona} is competitive with GP in this aspect: its $w_r$ bars are either better or no more than 1 order of magnitude worse than that of GP. In face of 20~dB noise, the GP $w_r$ markers nearly coincide with those in the noiseless case, showing high noise robustness. NSM comes as second, and lastly, \textsc{Diona} shows a marked difference in coverage, especially with Park and Friedman functions.


\begin{figure}[t]
\centering
\includegraphics[width=\textwidth]{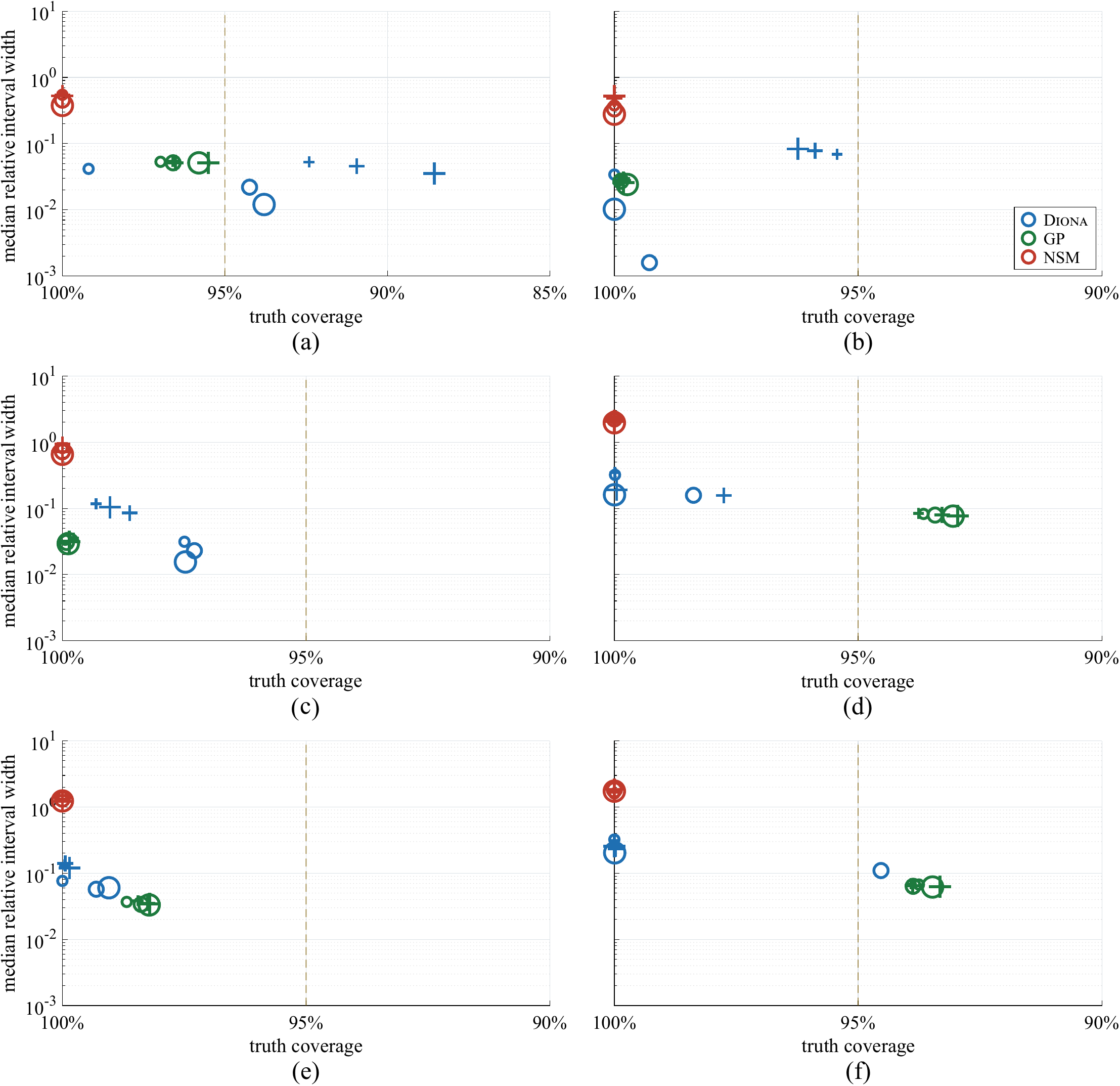}
\caption{Results: relative uncertainty interval vs truth coverage\\(a) Ishigami, (b) Park, (c) Friedman, (d) Hartmann, (e) Piston, (f) Rosenbrock.}
\label{fig:results-uncertainty}
\end{figure}

\subsection{Results: empirical computational burden}
Fig.~\ref{fig:results-computing-times} shows the computational burden $xy$-scatterplot for the inference (per-query) times vs training times, using the same marker conventions as the previous figure. The markers to the lower-left corner are better, meaning shorter training and per-query times. Note that while \textsc{Diona} has zero training time, its markers are placed to the leftmost axis for visualisation. Across all tested functions, there is no significant difference in having noiseless and noisy samples. In fact, it was also quite consistent how the relative ordering of the compared methods lines up w.r.t. the training times. Unsurprisingly, GP and NN take the rightmost markers (with training times in the order of minutes), followed by RBF and NSM. IDW has resulted in the best trade-off between per-query and training times, with markers towards the lower left of the plots. Lastly, \textsc{Diona} has unsurprisingly taken the most per-query time, since every query entails solving a small optimisation problem, while the others (except for NSM) merely involve evaluating a fitted (trained) model/surrogate function. IDW, NSM, RBF, and GP showed an increasing trend in both per-query and training times when presented with denser data sets, while \textsc{Diona} and NN showed no changes in computing times.


\begin{figure}[t]
\centering
\includegraphics[width=\textwidth]{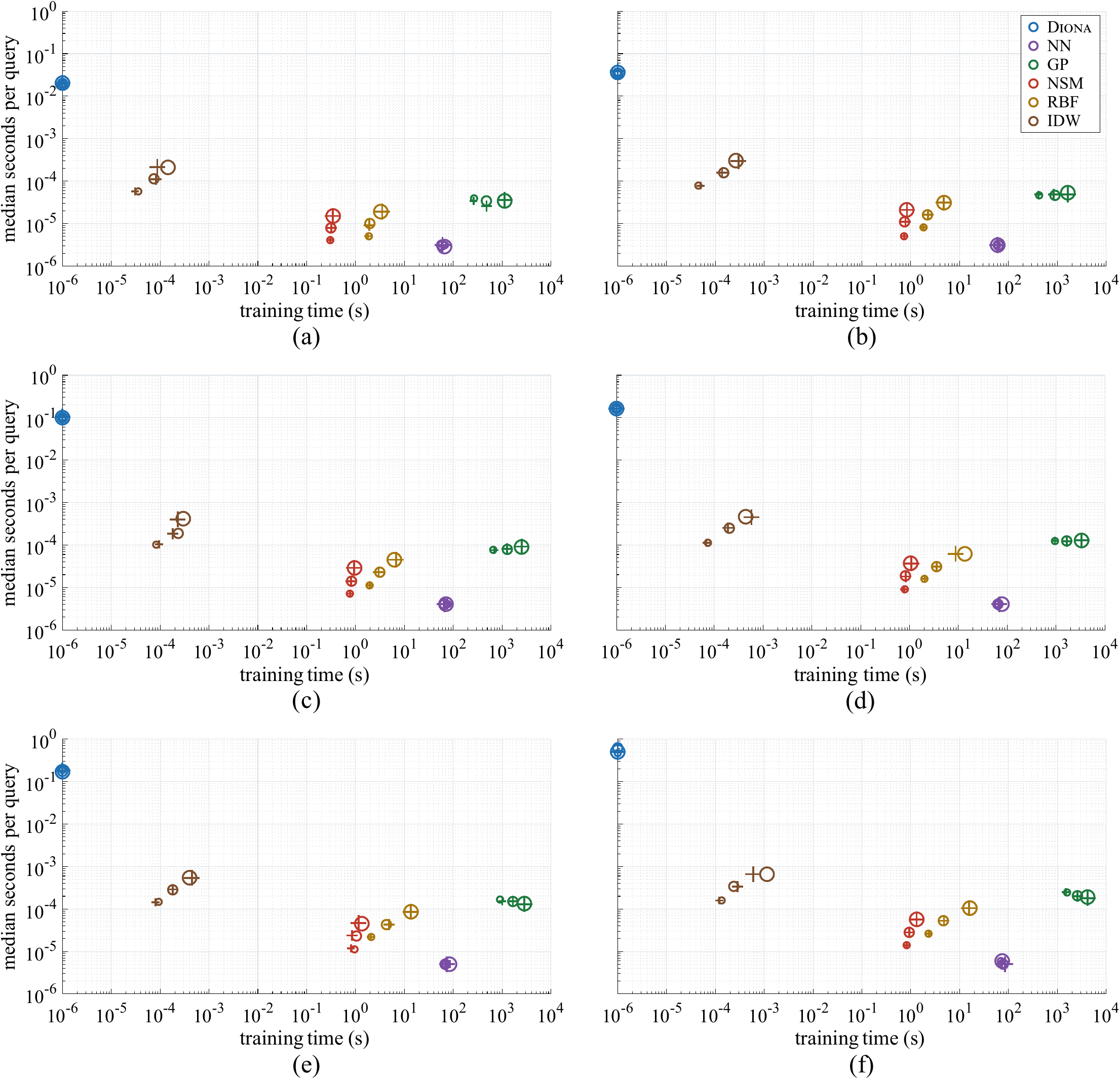}
\caption{Results: median per-query times vs training times\\(a) Ishigami, (b) Park, (c) Friedman, (d) Hartmann, (e) Piston, (f) Rosenbrock. \\ \textsc{Diona} (with zero training times) is placed to the left for comparative visualisation.}
\label{fig:results-computing-times}
\end{figure}

\subsection{Synthesis}

The above comparison to accuracy, UQ informativeness, and computational times point to an observed trade-off between our proposed \textsc{Diona} vs other methods. While \textsc{Diona} is arguably competitive w.r.t. the other compared algorithms, especially NN and GP, the true trade-off comes from the computational demands of the application. In particular, \textsc{Diona} can be better when the application fundamentally cannot afford a long global training loop, e.g., when data samples are introduced in real-time, and the predictions are only required for a few points, e.g., for rapid decision making. 

However, the disadvantage of such per-query methods is the lack of compressive representation of the data, meaning that as it stands, \textsc{Diona} needs the entire data set ready for fetching, imposing an increasing (and at times, unreasonable) demand for storage. Furthermore, the per-query solve time scales polynomially with $D$, while such a scaling does not apply to evaluating a pre-trained model. Therefore, for huge static data sets, it pays to perform a global training process arriving at a surrogate that is both a compressive representation, and a quickly-evaluable function. The user therefore needs to ponder on these trade-offs when choosing the most appropriate regressor for a particular application.

\section{Conclusions and outlook}

In this paper we introduced Direct Inference on Admissible Sets (\textsc{Diona}) that uses local information around a query point, together with Taylor relations-based constraints, to cast the estimation and uncertainty quantification of an unknown function as a series of linear programs. Such a formulation eliminates the need for a long, offline training stage, because the computational burden is paid per query. In a benchmark on six functions of three to eight dimensions, \textsc{Diona} was more accurate than GPs in the noiseless setting especially on the densest data sets. Furthermore, \textsc{Diona}, like GP, produced informative bounds containing the true values at more than 95\% of query points in four functions. We note an explicit trade-off on the online vs offline computational burden: while \textsc{Diona}'s time to a first prediction is orders of magnitude shorter, each \textsc{Diona} query costs more than evaluating a trained model. This low latency lends an advantage when relatively few queries are needed between data set updates.

Ongoing work includes a practicable hybrid offline-online approaches to address the practical computational trade-offs, reducing per-query computational burden and storage requirements, while maintaining the responsiveness to new samples. Furthermore, the inclusion of physics information, especially to extrapolate predictions outside the data set domain is an interesting avenue for future work.

\bibliographystyle{elsarticle-num-names}
\bibliography{diona-related-work}

@article{milanese2004setmembership,
  author   = {Milanese, Mario and Novara, Carlo},
  title    = {Set Membership Identification of Nonlinear Systems},
  journal  = {Automatica},
  year     = {2004},
  volume   = {40},
  number   = {6},
  pages    = {957--975},
  doi      = {10.1016/j.automatica.2004.02.002}
}

@article{martin2024polynomial,
  author   = {Martin, Tim and Allg{\"o}wer, Frank},
  title    = {Data-Driven System Analysis of Nonlinear Systems Using Polynomial Approximation},
  journal  = {IEEE Transactions on Automatic Control},
  year     = {2024},
  volume   = {69},
  number   = {7},
  pages    = {4261--4274},
  doi      = {10.1109/TAC.2023.3321212}
}

@article{roll2005direct,
  author   = {Roll, Jacob and Nazin, Alexander and Ljung, Lennart},
  title    = {Nonlinear System Identification via Direct Weight Optimization},
  journal  = {Automatica},
  year     = {2005},
  volume   = {41},
  number   = {3},
  pages    = {475--490},
  doi      = {10.1016/j.automatica.2004.11.010}
}

@article{calliess2020lacki,
  author   = {Calliess, Jan-Peter and Roberts, Stephen J. and Rasmussen, Carl Edward and Maciejowski, Jan},
  title    = {{Lazily Adapted Constant Kinky Inference} for Nonparametric Regression and Model-Reference Adaptive Control},
  journal  = {Automatica},
  year     = {2020},
  volume   = {122},
  pages    = {109216},
  doi      = {10.1016/j.automatica.2020.109216}
}

@article{milanese2011unified,
  author   = {Milanese, Mario and Novara, Carlo},
  title    = {Unified Set Membership Theory for Identification, Prediction and Filtering of Nonlinear Systems},
  journal  = {Automatica},
  year     = {2011},
  volume   = {47},
  number   = {10},
  pages    = {2141--2151},
  doi      = {10.1016/j.automatica.2011.03.013}
}

@misc{huang2023lipschitz,
  author        = {Huang, Julien Walden and Roberts, Stephen and Calliess, Jan-Peter},
  title         = {{Lipschitz} Interpolation: Non-parametric Convergence under Bounded Stochastic Noise},
  year          = {2023},
  howpublished  = {arXiv preprint arXiv:2310.07039},
  doi           = {10.48550/arXiv.2310.07039},
  eprint        = {2310.07039},
  archivePrefix = {arXiv}
}

@inproceedings{sabug2026rethinking,
  author        = {Sabug, Lorenzo and Kerrigan, Eric C.},
  title         = {Rethinking Physics-Informed Regression Beyond Training Loops and Bespoke Architectures},
  booktitle     = {Proceedings of the 23rd {IFAC} World Congress},
  year          = {2026},
  address       = {Busan, Korea},
  organization  = {IFAC},
  doi           = {10.48550/arXiv.2512.13217},
  eprint        = {2512.13217},
  archivePrefix = {arXiv},
  primaryClass  = {math.OC}
}

@article{novara2016setmembership,
  author   = {Novara, Carlo},
  title    = {Sparse Set Membership Identification of Nonlinear Functions and Application to Fault Detection},
  journal  = {International Journal of Adaptive Control and Signal Processing},
  year     = {2016},
  volume   = {30},
  number   = {2},
  pages    = {206--223},
  doi      = {10.1002/acs.2539}
}

@article{lu2008adaptive,
  author   = {Lu, George Y. and Wong, David W.},
  title    = {An Adaptive Inverse-Distance Weighting Spatial Interpolation Technique},
  journal  = {Computers \& Geosciences},
  year     = {2008},
  volume   = {34},
  number   = {9},
  pages    = {1044--1055},
  doi      = {10.1016/j.cageo.2007.07.010}
}

@book{bishop1995neural,
  author    = {Bishop, Christopher M.},
  title     = {Neural Networks for Pattern Recognition},
  publisher = {Oxford University Press},
  address   = {Oxford},
  year      = {1995},
  isbn      = {9780198538493},
  doi       = {10.1093/oso/9780198538493.001.0001}
}

@book{buhmann2003radial,
  author    = {Buhmann, Martin D.},
  title     = {Radial Basis Functions: Theory and Implementations},
  publisher = {Cambridge University Press},
  address   = {Cambridge},
  year      = {2003},
  series    = {Cambridge Monographs on Applied and Computational Mathematics},
  volume    = {12},
  isbn      = {9780521633383},
  doi       = {10.1017/CBO9780511543241}
}

@book{rasmussen2006gaussian,
  author    = {Rasmussen, Carl Edward and Williams, Christopher K. I.},
  title     = {{Gaussian} Processes for Machine Learning},
  publisher = {MIT Press},
  address   = {Cambridge, MA},
  year      = {2006},
  isbn      = {9780262182539},
  doi       = {10.7551/mitpress/3206.001.0001}
}

@book{theodoridis2020machine,
  author    = {Theodoridis, Sergios},
  title     = {Machine Learning: A {Bayesian} and Optimization Perspective},
  publisher = {Academic Press},
  edition   = {2},
  year      = {2020},
  isbn      = {9780128188033}
}

@book{trefethen2013approximation,
  author    = {Trefethen, Lloyd N.},
  title     = {Approximation Theory and Approximation Practice},
  publisher = {Society for Industrial and Applied Mathematics},
  address   = {Philadelphia, PA},
  year      = {2013},
  isbn      = {9781611972399}
}

@misc{surjanovic2013virtual,
  author       = {Surjanovic, Sonja and Bingham, Derek},
  title        = {{Virtual Library of Simulation Experiments}: Test Functions and Datasets},
  year         = {2013},
  howpublished = {Simon Fraser University},
  url          = {https://www.sfu.ca/~ssurjano/},
  note         = {Accessed 5 October 2026}
}

@article{dette2014eoptimal,
  author  = {Dette, Holger and Grigoriev, Yuri},
  title   = {{$E$}-optimal designs for second-order response surface models},
  journal = {The Annals of Statistics},
  year    = {2014}, volume = {42}, number = {4}, pages = {1635--1656},
  doi     = {10.1214/14-AOS1241}
}

\section*{Appendix}

\subsection*{Proof of Lemma~\ref{lemma:containment}}
\begin{proof}
    The true collection vector as $\bm{\psi}$ includes the true triple at query point $\bm{\theta}'$, along with those at the neighbouring samples. By Assumptions~\ref{assm:lipschitz-hessian}-\ref{assm:finite-noise}, trivially we have $\bm{\psi} \in \mathcal{A}$ by construction. Since $\admtriplesset$ is the projection of $\mathcal{A}$ onto the query-triple coordinates, and the query point block of $\bm{\psi}$ is $\bm{\theta}'$, we have $\bm{\theta}' \in \admtriplesset$.
\end{proof}

\subsection*{Proof of Lemma~\ref{lemma:worst-case-diameter-states-derivatives}}
\begin{proof}
    Given $\bm{x}'$ and $\bm{\theta}_1, \bm{\theta}_2 \in \admtriplesset$, then $\forall \idxpt{x}{i}{} \in \mathcal{N}$,

    \begin{align}
        \left| \idxval{z}{i}{} - \left(u_1 - \bm{a}_i^\top \bm{g}_1 + \bm{\lambda}_i^\top \bm{h}_1 \right)\right| &\leq \bar{\epsilon} + \frac{L_H}{6} \|\bm{a}_i\|^3 \\
        \left| \idxval{z}{i}{} - \left(u_2 - \bm{a}_i^\top \bm{g}_2 + \bm{\lambda}_i^\top \bm{h}_2 \right)\right| &\leq \bar{\epsilon} + \frac{L_H}{6} \|\bm{a}_i\|^3.
    \end{align}

    Subtracting the first and second inequalities through the triangle inequality, we have

    \[
        \left| \Delta u - \bm{a}_i^\top \Delta \bm{g} + \bm{\lambda}_i^\top \Delta \bm{h} \right| \leq 2 \bar{\epsilon} + \frac{L_H}{3} \|\bm{a}_i\|^3.
    \]

    Now we stack such inequalities for all selected neighbours $\idxpt{x}{i}{} \in \mathcal{N}$, and use the neighbourhood radius $r \doteq \max_{\idxpt{x}{i}{} \in \mathcal{N}} \|\bm{a}_i\|$ to build the normalised quantities $\hat{\bm{a}}_i$ and $\hat{\bm\lambda}_i$ as in Section~\ref{sec:neighbour-samples}. Substituting, and collecting the $r$-scalings onto $\Delta u, \Delta\bm g, \Delta\bm h$ rather than onto the coefficients, each stacked row becomes exactly $\bm\phi_i^\top \Delta$:

    \begin{equation}
        \left|\underbrace{\begin{bmatrix}
            1 & -\hat{\bm{a}}_{i_1}^\top & \hat{\bm\lambda}_{i_1}^\top \\
            1 & -\hat{\bm{a}}_{i_2}^\top & \hat{\bm\lambda}_{i_2}^\top \\
            \vdots & \vdots & \vdots \\
            1 & -\hat{\bm{a}}_{i_N}^\top & \hat{\bm\lambda}_{i_N}^\top
        \end{bmatrix}}_{\bm\Phi}
        \underbrace{\begin{bmatrix}
            \Delta u \\
            r\Delta \bm{g} \\
            r^2\Delta \bm{h}
        \end{bmatrix}}_{\Delta}\right| \leq
        \left(2\bar{\epsilon} + \frac{L_H}{3}r^3\right) \bm{1},
    \end{equation}

    \noindent an element-by-element inequality, where $\bm\Phi$ depends only on the normalised displacements $\hat{\bm{a}}_i$ and not on the neighbourhood radius $r$ (Section~\ref{sec:neighbour-samples}). Since $\mathcal{N}$ is $\mu_0$-poised (Definition~\ref{def:poised-neighbour-set}), $\bm\Phi$ is full column rank, $\bm\Phi^\dagger \bm\Phi = \bm{I}$, and it applies that

    \begin{align}
        \left\|\bm\Phi\Delta\right\|_2 &\leq \left\|\left(2\bar{\epsilon} + \frac{L_H}{3}r^3\right) \bm{1}\right\|_2 \\
        \|\Delta\|_2 &\leq \|\bm\Phi^\dagger\|_2 \|\bm\Phi \Delta\|_2 \\
        ~ &\leq \|\bm\Phi^\dagger\|_2 \left\|\left(2\bar{\epsilon} + \frac{L_H}{3}r^3\right) \bm{1}\right\|_2 \\
        ~ &= \left(2\bar{\epsilon} + \frac{L_H}{3}r^3\right) \|\bm\Phi^\dagger\|_2 \|\bm{1}\|_2
    \end{align}

    \noindent and we invoke that $\|\bm\Phi^\dagger\|_2 = \frac{1}{\underline{\sigma}(\bm\Phi)} \leq \frac{1}{\mu_0 \sqrt{N}}$ and $\|\bm{1}\|_2 = \sqrt{N}$, therefore

    \[
        \left\|\begin{bmatrix}
            \Delta u \\
            r\Delta \bm{g} \\
            r^2\Delta \bm{h}
        \end{bmatrix}\right\|_2 \leq
        \frac{1}{\mu_0}
        \left(2\bar{\epsilon} + \frac{L_H}{3}r^3\right).
    \]

    Finally, since each block of $\Delta$ has norm no greater than $\|\Delta\|_2$, we arrive at \eqref{eqn:worst-case-diam}, which concludes the proof.
\end{proof}

\subsection*{Proof of Lemma~\ref{lemma:dense-samples}}

\begin{proof}
    For this proof, we assume for simplicity and w.l.o.g. that the data set is augmented by one sample at every iteration. Consider that at iteration $n_0$ with the corresponding data set $\itrpt{X}{n_0}{}$, $\bm{x}'$ has the minimum neighbour set radius $r_{\mathcal{N}}^{\iter{n_0}*}$. We assume, for the sake of contradiction, that such radius will never change/shrink, i.e., as $n \rightarrow \infty,\; r_{\mathcal{N}}^{\iter{n}*} = r_{\mathcal{N}}^{\iter{n_0}*}$.

    Consider a smaller ball $\mathcal{B}(\bm{x}', r_s)$ where $r_s = \frac{1}{2} r_{\mathcal{N}}^{\iter{n_0}*}$, in which we place a reference points set $\bm{S} \doteq \left\{\bm{s}_m \in \mathcal{B}(\bm{x}', r_s), m = 1,\dots, N \right\}$ forming a $\mu_s$-poised configuration, i.e. $\underline{\sigma} (\bm{\Phi}(\bm S)) > \mu_s \sqrt{N}$ for some fixed $\mu_s > \mu_0$. Such a configuration exists since $\mu_0 < \bar{\mu}_D$: as $\bm{x}'$ is interior, $\mathcal{B}(\bm{x}', r_s)$ contains a smaller ball inside $\mathcal{R}$, into which any configuration can be scaled without changing $\bm{\Phi}$, as $\bm{\Phi}$ is normalised by the radius.

    We define for each $\bm{s}_m$ its respective closest sample
    \[
        \titrpt{s}{n}{m} \doteq \arg \min_{\idxpt{x}{i}{} \in \itrpt{X}{n}{}} \|\idxpt{x}{i}{} - \bm{s}_m\|,
    \]

    \noindent and collecting those points, we build $\titrpt{\Phi}{n}{}$. Now, due to the generation of dense samples, we have

    \[
        \forall \bm{s}_m \in \bm{S}, \;\; \|\titrpt{s}{n}{m} - \bm{s}_m\| \xrightarrow[n \rightarrow \infty]{} 0,
    \]

    \noindent and therefore $\| \titrpt{\Phi}{n}{} - \bm{\Phi}(\bm S) \| \xrightarrow[n \rightarrow \infty]{} 0$. By continuity of the singular values, $\underline{\sigma}(\titrpt{\Phi}{n}{}) \xrightarrow[n \rightarrow \infty]{} \underline{\sigma}(\bm\Phi(\bm S)) > \mu_s \sqrt{N} > \mu_0 \sqrt{N}$, so

    \[
        \exists M < \infty : \underline{\sigma}(\titrpt{\Phi}{n_0+M}{}) > \mu_0 \sqrt{N},
    \]

    \noindent and the corresponding samples set $\titrpt{S}{n_0+M}{}$ is $\mu_0$-poised with maximum radius at most $r_s + \max_m \|\titrpt{s}{n_0+M}{m} - \bm{s}_m\| \leq \frac{3}{4} r_{\mathcal{N}}^{\iter{n_0}*}$ (taking $M$ larger if needed), and therefore $r_{\mathcal{N}}^{\iter{n_0 + M}*} \leq \frac{3}{4} r_{\mathcal{N}}^{\iter{n_0}*}$, contradicting our assumption. We repeat the above arguments recursively, leading to

    \[
        r_{\mathcal{N}}^{\iter{n}*} \xrightarrow[n \rightarrow \infty]{} 0,
    \]

    \noindent which closes our proof.
\end{proof}

\subsection*{Proof of Theorem~\ref{theorem:convergence-true-triple}}

\begin{proof}
    By Lemma~\ref{lemma:containment}, $\bm{\theta}' \in \admtriplesset$. Furthermore, from Lemma~\ref{lemma:worst-case-diameter-states-derivatives} with $\overline{\epsilon}=0$, \eqref{eqn:worst-case-diam} reduces to
    
    \begin{align*}
        \| \bm{\theta}_1 - \bm{\theta}_2 \| \leq \frac{L_H}{3\mu_0} \sqrt{r^6_{\mathcal{N}} + r^4_{\mathcal{N}} + r^2_{\mathcal{N}}}
    \end{align*}
    
    \noindent for $\bm{\theta}_1,\bm{\theta}_2 \in \admtriplesset$. Lemma~\ref{lemma:dense-samples} (with Assumption~\ref{assm:poised-neighbour-set}), and replacing $\bm{\theta}_1$ with a generic triple $\bm{\theta} \in \admtriplesset$ and $\bm{\theta}_2$ with the true triple $\bm{\theta}'$, leads us to \eqref{eqn:convergence-true-triple}, closing the proof.
\end{proof}

\end{document}